\documentclass[10pt]{amsart}
\usepackage[pagebackref]{hyperref}
\usepackage{mathrsfs}
\usepackage{graphicx}
 \usepackage[all]{xy}
  \usepackage{amssymb}	
  \usepackage{color}
  \usepackage{longtable}
\usepackage{tikz}\usetikzlibrary{shapes.geometric,intersections,3d,quotes,angles}
\usepackage[pagebackref]{hyperref}
\usepackage{float}
\usepackage{mathtools}

\newtheorem{thm}{Theorem}[section]
	\newtheorem*{thm*}{Theorem}
	\newtheorem*{corr*}{Corollary}
	\newtheorem{lemma}[thm]{Lemma}
	\newtheorem*{lemm*}{Lemma}
	
	\newtheorem*{prop*}{Proposition}
	\newtheorem{corr}[thm]{Corollary}
	\newtheorem{crit}[thm]{Criterion}
	
	\newtheorem{obsv}[thm]{Observation}
	\newtheorem{summar}[thm]{Summary}
\theoremstyle{definition}
	\newtheorem{dfn}[thm]{Definition}
	\newtheorem*{dfn*}{Definition}
	\newtheorem{exmple}[thm]{Example}
	\newtheorem*{exmple*}{Example}
	
	\newtheorem{conj}[thm]{Conjecture}
	\newtheorem*{conj*}{Conjecture}

	\newtheorem*{cond*}{Condition}
	\newtheorem{proprs}[thm]{Properties}
\theoremstyle{remark}

	\newtheorem*{rmq}{\textit{Remark}}
       	\newtheorem{rmk}[thm]{\textit{Remark}}
 \usepackage[stix2]{newtxmath} 		% substitute  stix2-fonts for newtxmath gives nicer "v"
 	 \usepackage[cal=boondoxo]{mathalfa} 	% mathcal
	 \usepackage{physics} 
\newcommand{\bC}{{\vvmathbb{C}}}

\newcommand{\bN}{{\vvmathbb{N}}}
\newcommand{\bP}{{\vvmathbb{P}}}
\newcommand{\bQ}{{\vvmathbb{Q}}}
\newcommand{\bR}{{\vvmathbb{R}}}
\newcommand{\bZ}{{\vvmathbb{Z}}}
\newcommand{\mbold}[1]{\vb*{#1}} 
	
	\newcommand{\bE}{{\mbold{E}}}

	\newcommand{\bbP}{{\mbold{P}}}

	\def\ii{{\mbold i}}

	\def\b1{{\mbold 1}}

 \usepackage[cal=boondoxo]{mathalfa}

\newcommand\cP{{\mathcal P}}

  \DeclareFontEncoding{LS1}{}{}
    \DeclareFontSubstitution{LS1}{stix2}{m}{n}
              \DeclareSymbolFont{symbols2}{LS1}{stixfrak} {m} {n}
		\DeclareMathSymbol{\operp}{\mathbin}{symbols2}{"A8}
  		
 \def\mapright#1{\xrightarrow[]{#1}}

\def\into{\hookrightarrow}

	\DeclareRobustCommand\longtwoheadrightarrow
     {\relbar\joinrel\relbar\joinrel\twoheadrightarrow}
\def\lset{\{}  	% for { 
\def\rset{\}}  	% for } 
	\def\set#1{\lset#1\rset} 
	\def\sett#1#2{\lset #1 \mid  #2 \rset}  
	\renewcommand\setminus{-}			% new definition
\newcommand\slg[2]{\operatorname{\mathsf SL}_{#2}({#1})}

\def\alg{\operatorname{\mathsf Alg}}

\def\codim{\operatorname{codim}}
\def\comp{\raise1pt\hbox{{$\scriptscriptstyle\circ$}}}
\def\ext{\operatorname{Ext}}
\def\half{\frac{1}{2}}
\def\hdg{\text{\rm Hdg}}
\def\Hom{\text{\rm Hom}}
 \def\im{\text{\rm Im}}
 \def\id{\text{\rm id}} 							%identity
 
\def\pic#1#2{\text{Pic}^{#2}(#1)}
\def\prim{\text{\rm prim}}

\def\aj{\text{\rm AJ}}
\def\alb{\mathop{\rm Alb}\nolimits}

\def\alg{\text{\rm alg}}

\def\rat{\text{\rm rat}}
\def\inc{\text{\rm inc}}
\def\num{\text{\rm num}}
\def\cl#1{\mathsf{cl}(#1)} % class of cycles
\def\cyc#1#2#3{{\mathsf{Z}}^{#1}_{#2}(#3)}% voor cykels, gebruik: bv $\cyc k \num X$
\def\cors#1#2#3{{\mathsf{Corr}}^{#1}_{#2}(#3)}% voor correspondenties, als cykels

\def\pic{\mathop{\rm Pic}\nolimits}
\def\picmap{\mathop{\rm pic}\nolimits}

\def\ch#1#2#3{{\mathsf {CH}}^{#1}_{#2}(#3)}%  Chow groepen van varieteiten, gebruik: bv  $\ch k \num X$ of bv $\ch{}{} X$
\def\chow#1#2{\ch{#1}{}{#2}} % afkorting voor $\ch#1{}#2$
\begin{document}

\title{Incidence Equivalence, a survey.}% Version 2.}
\author{Chris Peters}
\address{University of Leiden and Technical University of Eindhoven}
\date{\today}

\begin{abstract} This is a survey of  results on incidence equivalence, a notion  introduced by P~Griffiths around 1970   
 when  trying to  extend  the classical properties of the Abel--Jacobi map for curves. 
Using the intermediate jacobians
and the associated Abel--Jacobi maps in higher dimension, a natural question came up:
is   the geometrically defined  incidence equivalence relation  the same 
as Abel--Jacobi equivalence, which is of  transcendental  nature?
I give an overview of results related to this question
and   to several classical conjectures that are far from resolved, such as Grothendieck's generalized Hodge conjecture. 
The motivation  for  writing   this survey came from  a recently observed  unexpected connection  of Griffiths' question 
to the asymptotic behaviour of the archimedean height pairing 
in a geometric setting.
  \end{abstract}
\maketitle
\section{Introduction}

\subsection*{Historical background}
The purely  geometric  notion of incidence equivalence has been introduced by P. Griffiths~\cite{griftrans} in order to generalize
 the classical properties of the Abel--Jacobi map
for divisors   to higher codimensions. Instead of using Weil's jacobians associated to a smooth projective variety, 
which are abelian varieties, he used his own version, the
so-called intermediate jacobians,   which are generally not projective, but have  the advantage of varying 
holomorphically in holomorphic families.
Griffiths posed the question whether for \emph{cycles algebraically equivalent to 0}   incidence equivalence is  the same as   
the transcendentally defined  Abel--Jacobi equivalence, at least up to torsion.
If so, this would lead to  a satisfying  extension  of Abel's theorem   (recalled in Example~\ref{exm:basic}) to   cycle classes
of higher  codimension. 
 He observed for instance that equality of the two equivalence relations would be implied by  Grothendieck's  generalized Hodge conjecture (which links algebraic and transcendental objects).
Independent of this conjecture equality was shown  for
 several classes of varieties such as smooth complete intersections and abelian varieties.
Since the articles~\cite{griftrans, grifschmid} from the early 1970s,  where this was discussed, it lasted until 1985 
when J. Murre~\cite{murre} proved the equality for codimension 2 cycles algebraically equivalent to 0. 
Inspired by this article and by Bloch's notion of biextension classes,
S. Müller-Stach ~\cite{mstach} took this up and showed  that equality of the two equivalence relations is intimately related to
the behaviour of  classes of Hodge level\footnote{See \S~\ref{sec:GHC} for the definition of Hodge level.}  $\le 1$ in odd-rank cohomology  under the Lefschetz decomposition.
  
  The interest in Griffiths' question   recently returned  when several people observed  that  the local geometric height pairing, 
  %an arithmetic concept, 
should be related to the 
archimedean height pairing, which is a Hodge theoretic concept. See for example~\cite{asympholms,naka,zhechen}. 
This relation  is a consequence of  two fundamental papers:
R. Hain's  result~\cite{hain}  on biextensions and   the  Brosnan--Pearlstein  Hodge theoretic results~\cite{bpjumps} on the asymptotic
behaviour of the height pairing for an admissible variation of mixed Hodge structure of biextension type.
These results apply in the case the biextension comes from two families of algebraically trivial cycles of 
complementary codimension.\footnote{This means here  
that the sum of codimensions of the  two cycles  is one more that the dimension of the fibre in which the cycles move.}
The results in the above papers then imply  that the archimedean height pairing  has   asymptotic behaviour of type $f(t)+ \alpha\log (|t|)$ 
where $t$ is any coordinate on the disc, $f$ is $C^\infty$  on the disc and  $\alpha$ 
  is a rational number independent of the choice of the coordinate.  
As recently shown by Z. Chen~\cite{zhechen}, if Griffiths' question 
has a positive answer,  $\alpha$ is the local geometric height pairing between the two cycles, which, as the  
terminology suggests, is indeed a geometric invariant.  
 
 Discussing with Z. Chen it became   clear that in view of these recent results an overview 
 of  the  classical  results related to incidence equivalence could be useful, and so I opted to write such  a review.
 An added motif to do so is that firstly,  some of the older literature 
 is apparently  not always easy to locate   and secondly, some
 arguments therein are based on knowledge quite common in the old days but seems to be forgotten by the new generation of algebraic geometers.

\subsection*{On the contents of this note}

 To set up  the required background,  the basics on Chow groups are summarized in \S~\ref{sec:chow}. For further reference, in this section 
 some attention is also given to the classical Hodge conjecture since its validity has important  consequences for all kinds of  widely believed
 conjectures and questions about algebraic cycles.
 
   Instead of morphisms, correspondences turn out to
 play a more natural role when studying algebraic cycle classes. For instance, 
  these are used to define the notion of 
 "incidence equivalence" in the title
 of this note.    I give a brief overview of their properties in  section~\ref{sec:cors}.  P. Griffiths' question is intimately related  
 to the so called  standard conjectures, not only to Grothendieck's  generalized   Hodge conjecture.
 For this reason a summary of these standard conjectures has been placed in \S~\ref{sec:gstconj}, based on 
 S. Kleiman's presentation~\cite{kleiman}.
  
  Before recalling in \S~\ref{sec:GHC}  how Griffiths' question is related to  all of this,   I recall in 
  \S~\ref{sec:AJ} the definition and the basic properties
  of the Abel--Jacobi maps.  The final section \S~\ref{sec:biexts} is devoted to a result of S. Müller-Stach~\cite{mstach}
in which  he reconsidered  Murre's result~\cite{murre} exploiting 
    a relation  with higher Chow groups via Bloch's biextensions.  For codimension $i$ cycles algebraically equivalent to 0
    this gives a criterion in terms of the cup-product pairing
    $H^{2i-1}(X)\times H^{2d-2i+1}\to H^{2d}(X)$ restricted to the $i$-th level of the coniveau filtration.
   For $i=2$   Murre showed  that this  restriction gives a   non-degenerate pairing   (cf. \cite[Lemma 5.2]{murre})
   and  gave a geometric argument  that this answers  Griffiths' question for $i=2$ positively.
   However his proof  does not seem to be complete. Fortunately, there is a rather straightforward approach based  on 
     observations  in Chapter IV of Weil's book~\cite{weilboek}
     which corrects this and even leads to    a  more general  result. Details are in   Corollary~\ref{cor:GenMurre}.
  \bigskip
  
  In the remainder of this note I shall use the following abbreviation.
  \begin{itemize}
 \item $GC(X,i) $:   Griffiths' question has a positive answer for codimension $i$ cycles on $X$.
\end{itemize}
  The main results in this overview are summarized  below. See also Summary~\ref{summar:main}, which  contains a few more technical
  results.

  \begin{thm*} Suppose $X$ is a smooth complex projective variety of dimension $d$. Then the following results hold.
    \begin{enumerate}
    \item   $GC(X,i) \iff 
 J^i_\alg(X)$ is isogenous to $\pic^i (X)$, the generalized Picard variety introduced in \S~\ref{sec:AJ}.
    Both would be true  if  Grothendieck's Hodge conjecture  $GHC (X,2i-1,i) $ holds, which is recalled below as Conjecture~\ref{conj:GHC}.
  \item  All primitive Lefschetz components of $H^{2i-1}_\alg(X)$ being  algebraic (in the sense of \S~\ref{sec:AJ})
  implies $GC(X,i)$.  This is the case if $B(X)$ holds (recalled in \S~\ref{sec:gstconj}).  In particular one has $GC(X,i)$ for abelian varieties $X$.
  \item $GC(X,1), GC(X,2) $  and $GC(X,d)$ hold.
 
  \item $GC(X,m)$ holds if $\dim X=2m+1$ and $b_{2k+1}(X)=0$ for $k\not=m$ such as for $2m+1$-dimensional complete intersections in projective space.
\end{enumerate}
\end{thm*}

\subsection*{Notational conventions}
 \begin{itemize}
\item If  $X$ is a manifold, $H^i(X)$ stands for singular  cohomology with integral coefficients, and if $F$ is a field, $H^i(X)_F =H^i(X)\otimes F$.
\item For a smooth projective variety, $\ch i {} X$ is the Chow group of $X$ (see \S~\ref{sec:chow}),  and $\ch i {} X_\bQ=\ch i {} X \otimes\bQ$.
\end{itemize}

  \medskip
 \subsection*{Acknowledgements} 
 \begin{small}
P. Griffths' Amsterdam 1971  lectures  on  cycles of higher codimension and the relation with Abel--Jacobi maps 
opened up a new world for me. As a beginning PhD student I only partly understood  the proofs, but Prof. Murre, then  one of my teachers, 
patiently explained what I was missing.  Thanks to Phillip for the beautiful mathematics he then explained,  and thanks  to my late colleague and friend
Jaap Murre to whom   I owe much of  my   knowledge concerning cycles and Chow groups.

 These notes are initiated by questions posed to me by Zhelun Chen. 
 I thank him for   not only providing  me with  relevant recent references but also for giving  a critical reading of a  coarse first draft. 
 His remarks and questions gave me a better feeling of what is no longer standard knowledge.
 
Thanks also to  Robin de Jong for communicating  some inaccuracies in my first drafts.
  \end{small} 

\section{Chow groups}

\label{sec:chow} 
Let $X$ be a smooth complex projective variety. I shall fix some notation related to cycles.
For definitions and basic applications the reader might consult \cite{LectMotives} and the older literature mentioned in it.
   
\begin{eqnarray*}
\cyc{i}{}  X &=&    \codim   i \text{ cycles on }  X   ,\\
\cyc{i}{\hom} X&=&  \codim   i \text{ cycles  on } X  \text{ whose homology class is zero,} \\
\cyc{i}{\num} X&=&  \codim  i \text{ cycles  on } X \text{ numerically equivalent to zero,} \\
\cyc{i}{\rat} X&=&  \codim   i \text{ cycles  on }  X \text{ rationally equivalent to zero,}\\
\cyc{i}{\alg} X&=&  \codim i \text{ cycles  on } X \text{ algebraically equivalent to zero.}
\end{eqnarray*}
Note that $\cyc{i}{\rat} X \subset\cyc{i}{\alg} X\subset  \cyc{i}{\hom} X\subset \cyc{i}{\num} X$.  The corresponding rational equivalence classes then
define the various Chow groups, denoted as follows.
$$
\begin{matrix}
 \ch i \alg X    =     \displaystyle \frac{ \cyc i \alg X }{\cyc{i}{\rat} X}   \subset  
 \ch i \hom X    =   \displaystyle \frac{ \cyc i \hom X } {\cyc{i}{\rat} X} \subset \hfill \\ 
\hspace{12em}\ch i \num X  =  \displaystyle \frac{ \cyc i \num X } {\cyc{i}{\rat} X}  \subset     \chow i X   =     \cyc{i}{}X / \cyc{i}{\rat} X.
\end{matrix}  
 $$
 Since $X$ is smooth projective, $\chow \bullet X$ is a  graded commutative ring under intersection of cycle classes and
 this is compatible with the ring structure on $H^\bullet_\bQ(X )$ under
 the class map $\text{cl}: \cyc{i}{}  X \to H^{2i}_\bQ(X )$, i.e. one has a commutative diagram
 $$
 \xymatrix{
 \chow  i X \times \chow j X  \ar[d]_{\rm cl} \ar[rr]^{\rm intersection} && \chow  {i+j} X  \ar[d]_{\rm cl}\\
 H^{2i}_\bQ (X)\times H^{2j}_\bQ (X)\ar[rr]^{\rm cupproduct} && H^{2(i+j)}_\bQ (X).
 }
 $$
  One sets
\begin{eqnarray*}
 H^{2i}_\alg  (X) &=&  \im (\text{cl})\subset H^{2i}_\bQ(X ), \\
 H^{2i}_\hdg(X)&=& H^{2i}(X)_\bQ\cap H^{i,i}(X).
\end{eqnarray*}
The classical Hodge conjecture  HC$(X,i)$   states that  $H^{2i}_\alg  (X) = H^{2i}_\hdg(X)$.
For the reader's convenience I'll summarize the recent status of this conjecture. See also
\cite[p. 91 and Appendix B]{lewisHC}.
\medskip

\noindent\textbf{Cases where HC$(X,i)$ holds, $X$  a smooth projective variety of dimension $d$:}
\begin{itemize}
\item HC$(X,1)$, HC$(X,d-1)$ (in particular it holds if  $d\le 3$).
\item HC$(X,2)$, $d=4$ and $X$ uniruled or unirational.
\item HC$(X,i)$ (all $i$) holds for $X$ a flag manifold.

\item  HC$(X,i)$ (all $i$) holds for $X\subset \bP^{d+r}$ a complete intersection  in the following cases 
\begin{itemize}
\item $X$ odd dimensional.
\item  $r=1, \deg X\le 5$.
\item $X$ Fano, $\deg X=4$.
\end{itemize}
\item For abelian varieties $X$ in case 
   \begin{itemize}
   \item $X$ sufficiently general,
   \item $X$ a product of elliptic curves,
   \item $X$ simple, $d=\dim X$ prime.
   \item  $\dim X=4$,  see \cite{markman}.
   \end{itemize}

\end{itemize}

% Lieberman \cite[Theorem 1]{lieb} has shown:
% \begin{thm}[Lieberman] Let $L,\Lambda,*$ be the standard operators on the cohomology of $X$, that is,
% \begin{itemize}
%\item $L:H^i(X,\bQ)\to H^{i+2}(X,\bQ)$ is cup product with a hyperplane class,
%\item $ *: H^i(X,\bQ)\mapright{\sim}{} H^{2d-i}(X,\bQ)$ is the star-operator,
% \item $\Lambda=*^{-1} L*: H^i(X,\bQ) \to H^{i-2}(X,\bQ)$.
%\end{itemize} 
%%
% The following statements are all equivalent: 
%  %
% \begin{enumerate}
%\item  \textbf{\emph{Conjecture $D(X)$}} holds, i.e.,  $\cyc {i}{\hom} X =\cyc{i}{\num} X$.
%\item  $\dim (H^{2i}_\alg (X))= \dim \mathsf  A^{d-i} (X)$ for all $i$.
%\item For all $c\in  H^{2i}_\alg (X)  $ with $c= \sum L^r c_r$,
% the Lefschetz decomposition, one has $c_r\in H^{2i}_\alg^{i-r}(X)$.
%\item If $c\in H^{2i}_\alg (X)$, then $*c\in H^{2i}_\alg(X)$.
%\item If $c\in H^{2i}_\alg (X)$, then $\Lambda(c) \in H^{2i}_\alg(X)$
%\end{enumerate}
%
% \end{thm}
% This is known in the following cases:
% \begin{itemize}
%\item $\cyc {1}{\hom} X =\cyc{1}{\num} X$ (divisors).
%\item $\cyc {2}{\hom} X =\cyc{2i}{\num} X$.
%\item $\cyc {d-1}{\hom} X =\cyc{d-1}{\num} X$ (curves).
%\item for $X$ an abelian variety  \cite{lieb}.
%\end{itemize}

\section{Correspondences}

\label{sec:cors}

For $X$ and $Y$ smooth projective, $\dim(X)=d$,  a degree $r$ \textbf{\emph{correspondence}}
from $X$ to $Y$  is an element of
$$
\cors r {} {X ,Y}  := \ch {d+r} {\bQ}  {X\times Y}.
$$
If $f\in \ch {d+r} { }  {X\times Y}$, one calls $f$ an \textbf{\emph{integral correspondence}}.
Integral  correspondences act on Chow groups and correspondences act on cohomology with coefficients in $\bQ$. 
 In order to define these actions, one uses 
the two projections $p:X\times Y \to  X$ and $q:X\times Y\to Y$ and the induced  (standard)  pull back actions
 $$
p^*:\ch {\bullet} { }  {X }
\to \ch {\bullet} { } {X\times Y} , \quad p^*: H^\bullet(X  )_\bQ\to H^\bullet (X\times Y)_\bQ ,
$$ 
both preserving degree, and the push forward actions 
$$
q_*: \ch {\bullet}  { }  {X\times Y} \to \ch {\bullet-d} { } Y,\quad 
 q_* : H^\bullet(X\times Y)_\bQ\to H^{\bullet -2d}(Y)_\bQ.
 $$ 
 The pull back  actions  give the  usual (contravariant) morphisms.
The  push forward map on cycle classes is the usual one, while the one on cohomology
  is the so called \textbf{\emph{Gysin map}} which is defined using 
 Poincaré duality as I'll explain now.  
  Let me first  recall  that  if $X$ is a compact connected oriented manifold of dimension $n$, with fundamental cohomology class $\cl X$,
 Poincaré duality for $H^\bullet(X)_F$,  $F$   a field, is the assertion that the cup product pairing
 \[
 H ^\bullet(X)_F \otimes H^{n-\bullet} (X)_F\mapright{\text{cup}} H^n (X)_F=\cl X\otimes F  \mapright{\simeq} F
 \]
 is non-degenerate. The linear map  $f^*: H^k(Y)_F\to H^k (X)_F$
  induced by  a continuous oriented map $f:X\to Y$, $\dim Y=m$,  has as its
 transpose the map $f_*: H^{n -k}(X)_F=H^k(X)_F^*   \to H^k(Y)^*_F=H^{m-k}(Y)_F $. 
 In other words, 
 $$f_*:H^\bullet(X) _F\to H^{\bullet+n-m}(Y)_F .
 $$
 From this description one sees that the action on cycles is compatible with the action on cohomology through the class map.
 
 \begin{rmq} With  $\bZ$-coefficients  one has the same map, 
but it is defined using the traditional Poincaré-duality isomorphisms  $P_X:H^k(X)\mapright{\sim} H_{n-k}(X) $
 and $P_Y :H^k(Y)\mapright{\sim} H_{m-k}(X)$ by the formula 
$$  f_!:=P_Y^{-1}\comp f_*\comp P_X,
$$
where $f_*$ is the map induced by $f$ in homology.
 The map  $f_!$  is   the (topological) Gysin homomorphism. 
 The first description is not possible here since   integral cohomology  has in  general torsion which one loses by taking duals.
 \end{rmq}

  Using the above, one defines  a (covariant) action of degree $r$ integral correspondences $f\in \cors r {} {X ,Y} $ on cycle classes  $\alpha$
 by   
 \begin{equation}
 \label{eqn:chowacts}
 f_*(\alpha)=q_*(p^*(\alpha)\cdot f) ,
  \end{equation} 
  where the dot stands for  intersection product.
So,    if $\alpha\in \chow i X$, one has 
$p^*(\alpha)\cdot f\in \chow {i+r+d} {X\times Y}$,
implying that $f_*(\alpha)\in \chow {i+r} Y$  and so
$$
f_*: \chow \bullet X\to \chow{\bullet +r} Y.
$$
  As for the action in cohomology, one 
uses the same formula~\eqref{eqn:chowacts} where now the dot  stands for cup product.
The same argument then gives the linear map
$$
   f_*:  H^{\bullet}(X)_\bQ \to H^{\bullet+2r}  (Y)_\bQ.
$$ 
Such linear  maps are called \textbf{\emph{algebraic linear  maps}}.

\begin{proprs} Let $X,Y $ be smooth projective of dimension $d$,$e$  and let $Z$ be smooth projective. 
\begin{itemize}
\item Correspondences can be composed: if $f\in \cors r {} {X,Y}$, $g\in  \cors s  {} {Y , Z}$, then $g\comp f= \text{pr}_{XZ} 
\left\{ (f\times Z)\cdot (X\times g)\right\} \in \cors {r+s}{} {X,Z} $.
\item If $f:X\to Y$ is a morphism, then its graph is an integral correspondence  $f_* \in \cors {d-e} {}{X,Y}$  and if $g:Y\to Z$ is a 
morphism, $g_*\comp f_*= (g\comp f)_* $, where the rightmost expression is the morphism induced by the usual composition of the morphisms
$f$ and $g$.
\item The cup product maps $L: H^\bullet(X)_\bQ\to H^{\bullet+2}(X)_\bQ$ are induced by the degree $ 1 $ correspondence
$p \times H\in \cors {1 }   {}{X,  X}=\ch  {d+1} {  } X_\bQ$, $p\in X$ a point, $H$ a smooth hyperplane section of $X$, and hence these are algebraic linear maps. 
\end{itemize}
\end{proprs}
 
\medskip

 Correspondences are used to define incidence equivalence, introduced by Griffiths \cite[p. 7]{griftrans}:

\begin{dfn}    \label{dfn:inc} $\alpha \in  \ch i \hom X$ is \emph{\textbf{incidence equivalent to zero}}, written $\alpha\sim_{\rm inc} 0$,
 if  $f_*(\alpha)=0$   for all  integral correspondences  $f\in \cors {1-i}{}{X,Y}$ acting on $ \ch i \hom X$, and where   $Y$ is any   smooth 
projective variety. 
\end{dfn}
To explain the terminology,  note that $f_y= f\cdot (X\times y)$, $y\in Y$,  is   a family of cycle classes     in $X$ parametrized by $Y$.
On the other hand, $f_*(\alpha)= \text{pr}_{XY} (f\cdot (\alpha\times Y))$ is the class of a  divisor $D_{\alpha,f}$    on $Y$. 
Passing to representative cycles  on $X\times Y$,
(without changing notation, but chosen in such a way in  their  rational equivalence classes so that  $ \alpha\times Y $ meets $f$ transversally),
observe that  $y\in Y$ lies on the support of $D_{\alpha,f}   $ if and only if for some $x\in X$ the pair $(x,y)$ lies on the support of $f$, i.e.,
 the support of   $f_y$ meets  the support of $D_{\alpha,f}   $. 
For this reason $D_{\alpha,f}$ is called the \emph{\textbf{incidence divisor}} associated to  the pair $(\alpha, \set{f_y})$. So $\alpha\sim_\inc 0$
if   these  incidence divisors  (for  $\alpha$  fixed) are rationally equivalent  to zero.
%
%\begin{rmk} Note that the target of $f$ is the Picard torus $\ch 1 \hom Y= \pic^{(0)} Y$. This gives the crucial link
%to the Abel-Jacobi maps, as  explained  in Section~\ref{sec:AJ}; see in particular   diagram~\eqref{eqn:CorOnJacs}.
%As an example  take $ i=1$.  Then $\alpha$ is the class of a divisor and $f_y$ is a family of $0$-cycles. Then $D_{\alpha,f}$
%
%I shall using the above convention for denoting cycles in the corresponding classes.
%Ten for a divisor $W$ of $X$ a cycle $Z \in Z^d(X\times C)$ meets $W\times C$
%in $Z\cdot (W\times C )$ which projects to $ Z(W)$, a cycle in $C$. If $W$ is linearly equivalent to $0$, the so is $\Gamma(W)$.
%So $W$ is incidence equivalent to $0$ if $\Gamma(W)\in \pic^{(0)}(C)$ for all such pairs $(C,Z)$ which implies that $W$ is 
%Abel--Jacobi equivalent to zero.
%\end{rmk}

\section{On the  Standard Conjectures}
\label{sec:gstconj}

 These conjectures have been formulated around 1964 by Grothendieck
for smooth projective varieties over any closed field $k$. These are in terms of
a Weil cohomology theory with  coefficients in a field $F$. See \cite{kleiman} and  \cite[\S 3.1]{LectMotives} for details on this.

In this survey only the case $k=\bC$, $F=\bQ$  plays a role 
and the Weil cohomology is just singular cohomology.   Part of these conjectures concern the properties of the Lefschetz operator
which acts on any Weil cohomology but  thanks to Hodge theory the Lefschetz operator behaves better if $k=\bC$.

 \subsection{Lefschetz' theory} 
 
 Let $X$ be a  complex  projective variety  of dimension $d$ and $L$ the Lefschetz operator.
The  \textbf{\emph{(strong) Lefschetz theorem}} states that one has isomorphisms
\[
 L^i :H ^{d-i}(X)\mapright{\sim} H^{d+i}(X),\quad i=0,\dots, d , 
\]
while the \textbf{\emph{weak Lefschetz theorem}} states that for an  embedding $\iota:Y\into X$  of a smooth hyperplane section,
one has
 $$
H^i(X) \mapright{\iota^*} H^i(Y) \text{ is }\begin{cases}
\text{an isomorphism} & \text{ for } i<d-1,\\
\text{injective} & \text{ for } i=d-1.
\end{cases} 
$$
 These two theorems are not known to hold for  other Weil cohomology theories, and for these they are   part of the standard conjectures.
 \par
The operator $\Lambda: H^{i} _\bQ (X)\to H^{i-2}_\bQ   (X) $ is a quasi-inverse for $L$   in the sense that using 
  the isomorphisms $L^\bullet : H^{d-\bullet} (X)_\bQ\to H^{d+\bullet}_\bQ (X) $ from the strong  Lefschetz theorem, 
 the following diagrams  commute,  defining $\Lambda$:
 \begin{eqnarray}
 \label{eqn:Lambda}
 \begin{split}
\begin{matrix}
\text{ for } \ell=0,\dots,d-2                      &                   & \hspace{-8em}  \text{ for } \ell=2,\dots, d             \\
\xymatrix{  H^{d-\ell}(X)_\bQ \ar[d]_\Lambda \ar[r]^{L^\ell} _\sim& H^{d+\ell}(X)_\bQ\ar[d]_L \\
 H^{d-\ell-2}(X)_\bQ \ar[r]_{L^{\ell+2}}^\sim& H^{d+\ell+2}(X)_\bQ
 }                                               &       &  \hspace{-8em}    \xymatrix{
                                                              H^{d-\ell+2} (X)_\bQ     \ar[r]^{L^{ \ell-2}}_\sim & H^{d+\ell-2}(X)_\bQ     \\
                                                             H^{d-\ell}(X)_\bQ \ar[u]_{L}    \ar[r]_{\sim}^{L^{\ell}} & H^{d+\ell}(X)_\bQ \ar@{->}[u]_{\Lambda},
                                                            }
\\
               \hspace{16em}           \xymatrix{
                                       H^{d-1}(X)_\bQ \ar@<1ex>[r]^L_{\sim}  &   H^{d+1}(X)_\bQ \ar@<1ex>@{->}[l]^{\Lambda}
                                      }
\end{matrix}\end{split}
 \end{eqnarray}
The cohomology has a  \textbf{\emph{Lefschetz decomposition}}
 involving primitive cohomology $H^{d-j}_\prim(X)= \ker (\Lambda: H^{d-j}(X)_\bQ\to H^{d-2-j}(X)_\bQ)$,
$j=0,\dots,d $, namely 
\begin{equation}
\label{eqn:lefschdec}
H^{i} (X)_\bQ= \bigoplus_{\ell=i_0}^{\lfloor i/2\rfloor }  L^\ell H^{i-2\ell}_\prim(X) , \quad i_0=\max(i-d,0)  .
\end{equation} 
As in \cite[\S 2.3]{3authorsBIS}, using diagrams~\eqref{eqn:Lambda} one  deduces that the total $\bQ$-cohomology
 consists of a direct sum of disjoint strings
$$
W_\ell= \left(\xymatrix{
                                       H^{d-\ell}_\prim(X) \ar@<1ex>[r]^-L_-{\simeq}  &   LH^{d-\ell }_\prim(X)  \ar@<1ex>@{->}[l]^-{ \Lambda} 
                                           \ar@<1ex>[r]^-L_-{\simeq}
                                        &    \cdots   \ar@<1ex>@{->}[l]^-{\Lambda} \ar@<1ex>@{->}[r]^-{L}_-{\simeq}   &  \ar@<1ex>@{->}[l]^-{\Lambda}  
                                        L^{\ell-1}H^{d-\ell }_\prim  (X)\ar@<1ex>[r]^L_-{\simeq}  &   
 L^{\ell } H^{d-\ell}_\prim(X) \ar@<1ex>@{->}[l]^-{\Lambda}
                                   }   \right)   .
$$
of lengths  $\ell+1 $, $\ell=0,\dots,d$,  and $L$ and $\Lambda$ are inverses of each other on each string of length $\ge 2$.   
In particular, if  $x=\sum_{a=s}^{s+t}   L^a x_{i-2a} \in H^i(X)_\bQ $ is the primitive decomposition of $x$,  the
summand $L^{s+t}x_{i-2s-2t}$, with maximal power of $L$ is  the only summand which produces a non-zero primitive vector 
under left composition with   $\Lambda^{s+t }$, i.e., $\Lambda^{s+t }x= x_{i-2s -2t}$. Subtracting this summand from $x$
gives  $y=x-L^{s+t} x_{i-2s-2t}$ with   highest power $L^{s+t-1}$ of $L$ in its decomposition and the same procedure gives
\begin{eqnarray*}
x_{i-2s-2t+2} = \Lambda^{s+t-1 } y& = &\Lambda^{s+t-1 } (x- L^ {s+t}  \Lambda^{s+t }   )x\\
&=&
 \Lambda^{s+t-1} \comp(\id- L^{s+t } \comp \Lambda^{s+t}) x.
\end{eqnarray*} 
Continuing in this way  one  finds a (in general non-commutative) polynomial $\Phi_{i,a}(L,\Lambda)\in \bZ[L,\Lambda]$ such that\footnote{See also \cite[Th\'eor\`eme 3, \S I.4]{weilboek}.}
%is the Lefschetz decomposition~\eqref{eqn:lefschdec} of $x$, each nonzero summand  $L^a x_{i-2a}$ belongs to a certain $W_\ell$,
%since $a=k$, $i-2a=d-\ell$ gives $\ell=d-i+2a$ and so 
\begin{equation}
\label{eqn:OnLambda}
x_{i-2a} =     \Phi_{i,a}(L,\Lambda)x,
\end{equation}   
which will be used below, for example to prove Lemma~\ref{lem:OnConjBeven}.

Over $\bC$  another standard conjecture,  the \textbf{index conjecture}  $I^i(X,L)$ is true. It asserts that the  product pairing on primitive algebraic classes  
$$
Q: H^{2i}_{\alg,\prim} \times H^{2i}_{\alg,\prim}\mapright{\quad} \bQ,\quad Q( x,y)= (-1)^i L^{d-2i} x\cdot y 
$$ 
 is positive definite.  This is already the case
 if  $x,y$ are primitive Hodge classes of type $(i,i)$, as  a consequence of the Hodge--Riemann bilinear relations ~ \cite[Theorem 1.33]{mht}
which I now recall.
 
 \begin{thm} \label{thm:HR} Let $X$ be a  complex projective manifold of dimension $d$ and let
 $$
 H^r_\prim (X)_\bC=\bigoplus_{p+q=r}H^{p,q}_\prim (X)
 $$
 be the Hodge decomposition of primitive cohomology, 
 \begin{equation}\label{eqn:Weil}
  C:H^r (X)_\bR\to H^r(X)_\bR, \qquad C|_{H^{p,q}_\prim (X)}=\ii^{p-q},
 \end{equation} 
  the Weil operator and 
  \begin{equation}\label{eqn:HR}
 Q(x,y)= (-1)^{\half r(r-1)} L^{d-r} x \cdot   y ,\quad x,y\in H^r_\prim (X)_\bC
 \end{equation}
 the bilinear \textbf{Hodge--Riemann  form}.
% This can be extended to the full cohomology roup $H^i(X)_\bR$, extending $Q$ by
%\begin{equation}\label{eqn:HR2}
% Q(x,y)= (-1)^{\half i(i-1)} \sum _r (-1)^r Q(x_{i-2r}, y_{i-2r}) ,
% \end{equation}
% where $x=\sum L^rx_{i-2r}$, $y=\sum L^r y_{i-2r}$  are  the primitive decompositions.
 Then the two Hodge--Riemann relations are valid
 \begin{description}
\item[Q(1)] $Q(H^{p,q}_\prim(X),H^{r,s}_\prim(X))=0, \text{  if } (p,q)\not=(s,r)$,
\item[Q(2)] $Q(Cx,\bar x) \ge 0$, for $ x\in H^{k}_\prim(X) $,   and $Q(Cx,\bar x) =x\cdot x >0$ if $x\not=0$,
 where $x\cdot x=\int_X x\cup x$, the standard "intersection" product. 
\end{description}
 \end{thm}

\subsection{Remaining standard conjectures}  
There are further conjectures playing a role below which are not known to be true in general, even over $\bC$:
\begin{itemize}

\item \textbf{Conjecture}     $A(X,L)$: $L^{d-2i}:H^{2i}_\alg(X) \to H^{2d-2i}_\alg(X)$ is an isomorphism.
 
 \item \textbf{Conjecture}   $B (X)$: The linear maps  $\Lambda $ are algebraic,  that is, they are induced by a self-correspondence 
(necessarily of  degree $-1$). 
  
  \item \textbf{Conjecture}  $ D(X)$: $\cyc {i}{\hom} X =\cyc{i}{\num} X$.
\end{itemize}

\medskip 
As to the status of the conjectures see~\cite{kleiman}:
\begin{lemma} 
\begin{itemize}
  
\item Over $\bC$ the conjecture  $A(X,L)$ follows from the Hodge conjecture and then $A(X,L)\iff D(X) $
(see \cite[3.9 Corollary]{kleiman}).

\item   $ B(X)$ does not depend on the chosen hyperplane class to define $L$ by~\cite[2.10 Corollary]{kleiman}.

\item   $ B(X)\implies A(X,L)$ by~\cite[2.2 Corollary]{kleiman}.

\item  $B(X)$ as well as $D(X)$ hold  if $X$ is an abelian variety by~\cite[Theorem 4]{lieb}, \cite[2. Appendix]{kleiman}.
\end{itemize}

\end{lemma}

An easy consequence of the above results on the Lefschetz decomposition is the following result.

\begin{lemma}  \label{lem:OnConjBeven}  Suppose that  $B(X)$ holds, then if $x\in H^{2i }_\alg(X) $, 
   all primitive components of the Lefschetz decomposition of $x$  belong to $H^{2\bullet}_\alg(X)$.
  \end{lemma}
  \begin{proof}  
 The correspondences $L$ and $\Lambda$ being algebraic,
  the same holds for the polynomials  $\Phi_{i,a}(L,\Lambda)$ in 
  equation~\eqref{eqn:OnLambda}. So this equation 
implies  that  if $x$ is algebraic, the primitive components $x_{2i-2a}  $ of the Lefschetz decomposition of $x$ are all algebraic.
   \end{proof}

\section{Abel-Jacobi maps} 

\label{sec:AJ}

\subsection{Tori revisited} First I recall some elementary facts about complex tori $T=V/\Gamma$, where $V$ is the universal cover of $T$ viewed 
as a complex vector space, and $\Gamma$ is a maximal integral lattice of $V$. The \textbf{\emph{dual torus}}  by definition is equal to
$T^*=\bar V^*/\Gamma^*$, where the bar denotes complex conjugation,  $V^*$ the dual space of $V$,
and $\Gamma ^*=\Hom_\bZ(\Gamma,\bZ)$ the dual lattice. If $T$ is an abelian variety, so is $T^*$ and then $T^*=\pic ^0 T$
which is the set of homological trivial  divisor classes  on $T$ with the addition operation.

A holomorphic map  $f:T_1=V_1/\Gamma_1\to V_2/\Gamma_2$ beween complex tori  sends the origin of $T_1$ to a point $a\in T_2$
and such a map  is then of the form $t_a\comp   F$, where $t_a$ is the translation $x\mapsto x+a$ and $  F$
 is  induced by a complex linear map  $\tilde F:V_1\to V_2$. These induce homomorphisms between the tori, but even if $\tilde F$ is an isomorphism,
 $f$ might not be an isomorphism, since the induced map on  the lattices need not be an isomorphism. Indeed, in general
  $f$  is only an isogeny. This  method of gathering information  on $f$ is applied to obtain several of the main results  in this note and
  this   is the underlying reason why one can only show that such $f$ are isogenies   instead of isomorphisms.   
 
\subsection{Tori associated to odd weight Hodge structures} In case $V$ carries  an integral odd weight    Hodge structure $(V, F^\bullet )$, there are two  canonical ways to associate
a complex torus to it using operators defined by the Hodge decomposition $V=\oplus V^{p,q}$.
    The first is  the Weil operator (see equation~\eqref{eqn:Weil}) which in this case is given  by
 $C|_{V^{p,q}}= \ii ^{p-q} $ for which $C^2=-\id$ since  the weight $p+q$ is odd.  This is a real operator and so it gives 
 $V_\bR$ a complex structure and yields  Weil's jacobian $J_W(V)$ (see \cite[p. 82]{weilboek}). 
  Weil further  shows (loc. cit.)  that  the Lefschetz decomposition gives a polarizing form on $J_W(V)$
 and so gives an abelian variety.   If one however uses the operator $C'$ which is multiplication by $\ii$ for $p>q$ and by $-\ii$ if $p<q$ the resulting complex
 torus is Griffiths'  intermediate jacobian $J(V )$ (see \cite[\S 2]{grifbombay}) whch in general is not  an abelian variety.
  If the Hodge decomposition  has only two components,  clearly $C=C'$ and hence:
 
 \begin{obsv} \label{obs:JacIsAb}  Griffith's intermediate jacobian $J(V )$ is an abelian variety in case $V_\bC= V^{i,i-1} \oplus V^{i-1,i} $,
 i.e., if the level of $V$ is $\le 1$, where \emph{level}$=\max_{V^{p,q}\not=0} |p-q|$.
 \end{obsv} 

 These constructions can be  applied to odd rank cohomology groups $H^{2i-1}(X)$ of complex projective manifolds $X$ and 
 give $J^i(X)_W$, in the Weil case, respectively $J^i(X)$ in Griffiths' case.
  If $X$ is a  complex projective manifold $X$  the latter  can be described in an alternative fashion as
$$
  J^i(X)  = H^{2i-1}(X)_\bC  / [F^i H^{2i-1} (X)+\Im(H^{2i-1} (X)\to H^{2i-1}(X)_\bC)],
$$
 where $F^\bullet H^{2i+1}(X)_\bC$ is the Hodge filtration of the standard Hodge structure.
 So here the universal cover is  $  H^{2i-1}(X)_\bC/ F^i H^{2i-1}$ and by Poincaré duality
 \[
 ( \overline{H^{2i-1}(X)_\bC/ F^i H^{2i-1}})^*  = H^{2d-2i+1}(X)_\bC/F^{d+1-i}
 \]
 and so 
 $$
   \text {  the dual of } J^i(X) \text{ is } J^{d+1-i}(X) .
 $$
  
 \subsection{On the Abel--Jacobi map}  
 
 Since the cup product  
 $$ H^{2i-1} (X) _\bC \otimes H^{2d-2i+1}(X)_\bC\to \bC
 $$ is zero when restricted to $F^i\otimes F^{d-i+1}$
one has an induced  linear isomorphism  
 $$
 H^{2i-1}(X) _\bC/F^i \mapright{\simeq}  (F^{d-i+1}H^{2d-2i+1}(X))^*.
 $$ 
 It  sends the image of $H^{2i-1}(X)$ in $H^{2i-1}(X)_\bC$ to the image of the dual of $H^{2d-2i+1}(X)$ in the target  of this isomorphism
 and so gives yet another description of $J^i(X)$ as the dual of $J^{d+2-i}(X)$.  This 
 is used to define the Abel--Jacobi map
 \begin{eqnarray*}
 \psi^i_{\rm AJ}:  \ch i \hom X & \to & J^i(X)\\
 \alpha  &\mapsto&  \int_\gamma \omega  \pmod {H^{2i-1} (X))}.
\end{eqnarray*} 
 More explicitly, $\alpha\in  \ch i \hom X$ is sent  to the   functional  $f_\alpha : F^{d+1-i}H^{2d-2i+1} (X)\to \bC$ 
  defined as  $f_\alpha (\omega)=  \int_\gamma \omega$, where $\omega$  is a closed $2d-2i+1$-form whose class belongs to
  $F^{d+1-i}H^{2d-2i+1} (X)$, and $\gamma$ is  a chosen
   $2d-2i+1$-chain  with  $  \alpha=\partial \gamma$. 
   The ambiguity that this choice brings   is compensated by
   dividing out by the dual of the lattice $H^{2d-2i+1} (X)$ modulo  torsion.
% To see the relation with numerical equivalence consider the diagram
% \[
% \xymatrix{ \ch i\hom X \ar[d]^{\psi_\aj^i} \hspace{-2em}   &\times     &\hspace{-2em}    \ar[d]^{\psi_\aj^{d+1-i}}  \ch {d+1-i} \hom X \ar[rr] ^{\text{intersection}} &&  \ar[d]_{\psi_\aj^{d+1}} \ch d \hom X\\
%                                        J^i(X)             \hspace{-2em}      &\times    &   \hspace{-2em}     J^{d+1-i}(X)                                                 \ar[rr]  &  &   \bC^\times }
% \]
%If %$\psi_\aj^i (\alpha)\cdot \psi_\aj^{d+1-i}=\psi (\alpha\cdot \beta)=0$
%for $\alpha\in \ch i \hom X$ one has $\alpha\cdot \beta=0$,
% for all $\beta\in   \ch {d+1-i} \hom X$, then $\alpha\sim_\num 0$.
 This leads to an equivalence relation:

\begin{dfn} A cycle class $\alpha\in \ch i \hom X$ is \emph{\textbf{Abel-Jacobi equivalent to zero}} 
 %, written $\alpha\sim_{\rm AJ}0$ 
if it is in the kernel of $ \psi^i_{\rm AJ}$. 
The subgroup it defines is denoted $  \ch i {\rm AJ}  X$.
\end{dfn}

\begin{exmple} \label{exm:basic}
 The theorem of Abel for divisors classes on a curve   states that the associated Abel--Jacobi map  
 $\psi_\aj^1 : \ch 1 \hom    X\to \pic ^{(0)} X$ is an isomorphism which implies
 in particular that incidence equivalence coincides with Abel--Jacobi equivalence. For divisors on higher dimensional projective varieties 
 this is also the case as demonstrated in a transcendental fashion by K. Kodaira (cf.~\cite[Theorem 7]{kodGreen}) and 
 later by D. Lieberman in ~\cite{liebIJ}  who paved the way for the algebraic approach of H. Saito in~\cite{hsaito}.

Dually, for $0$-cycles one has the Albanese variety $\alb (X)$ and the Albanese map $\psi_\aj^d:\ch d \hom X \to \alb (X)= J^d(X)$.
However the Albanese kernel  $\ch d \aj X$ is in general non-zero but since the Albanese variety is dual to the Picard variety, incidence
equivalence coincides with Albanese equivalence. See Corollary~\ref{cor:OmDualPics}. 
\end{exmple} 

Since correspondences act on cohomology, they also induce maps on the level of intermediate jacobians.
Here one restricts to cycles  algebraically equivalent to $0$ and one sets
$$
 J^k_\alg(X):= \psi^k_{\rm AJ}(\ch * \alg {X} ).
 $$  
By \cite[Theorem 12.24]{lewisHC} an integral correspondence $   f  \in \cors r {} {X ,Y}$ induces  a  homomorphism  
\begin{equation*}%\label{eqn:CorOnJacs}
  f_* : J^i _\alg (X) \to J^{i+r}_\alg (Y),
\end{equation*} 
fitting in  a commutative diagram involving Abel-Jacobi maps:
\begin{equation}
\label{eqn:CorOnJacs}
 \begin{split}
 \xymatrix{
  \ch i  \alg  X \ar[r]^{f_*} \ar@{->>}[d]_{\Psi^i_{\rm AJ}}&  \ch {i+r}  \alg  Y  \ar@{->>}[d]_{\Psi^{i+r}_{\rm AJ}}\\
J^i_\alg(X) \ar[r]^{f_*} & J^{i+r}_\alg(Y).
}
\end{split}
\end{equation}
 In the special case of incidence equivalence $r=1-i$ and so 
the diagram above  reads
\begin{equation*}
 \xymatrix{ 
\ch i \alg X \ar[r]^{f_*} \  \ar@{->>}[d]_{\Psi^i_{\rm AJ}}&  \ch {1}  \alg Y \ar[d]_{\simeq}\\
\hspace{4em}J^i_\alg(X)\subset J^i(X) \ar[r]^-{f_*} & J^{1}(Y)=\pic^0(Y).
}
 \end{equation*}
From this diagram one deduces that
\begin{equation}
\label{eqn:CruxInc}  \ker(\Psi^i_{\rm AJ} ) \cap \ch i  \alg  X \subset K^i(X):= \sett{\alpha\in \ch i \alg X}{\alpha\sim_\inc 0}.
\end{equation}
Problem B  in \cite{griftrans} is known as 

\medskip
\noindent 
$GC(X,i)$: \emph{after tensoring with $\bQ$, the inclusion \eqref{eqn:CruxInc} is an equality, i.e., 
on classes in $\ch i \alg X \otimes   \bQ$  incidence equivalence coincides with Abel--Jacobi equivalence.}
 
  \subsection{Further properties of the Abel--Jacobi map}
 There is   a version of the Abel-Jacobi map for families of cycles   on $X$ within  the same cohomology class.
More precisely, \emph{a family of codimension $i$ cycles} is a cycle $Z$ on $Y\times X$ such that 
 $Z$ meets the fibers $\set y \times X$ 
of the projection $Y\times X\to Y$ in a cycle $Z_y$ where $i=\codim Z_y$ is independent of $y$.
Supposing that $\cl{Z_y}$ is independent of $y$,   the class of the cycle  $Z_y-Z_{y_o}$ is homologous to $0$  and so  one has an Abel-Jacobi map   
\[
 \psi_Z^i: Y \to\ch i \hom X \mapright{\Psi^i_\aj}
  J^i(X) ,\quad \psi_Z^ i(y)= \psi^i_\aj(\alpha_y), \, \alpha_y= [Z_y- Z_{y_o}].
\]
This morphism  has two crucial properties~\cite[Lemma 4.7.1]{3authorsBIS}:

\begin{lemma} \label{lem:OnHs} Let  $Z=\sett{Z_y}{y\in Y}$ be a family of codimension $i$ cycles on $X$, all
in the same homology class.
\begin{enumerate}
\item The Abel-Jacobi map $\psi^i_\aj$ is \textbf{a regular homomorphism}, i.e. for all $Z$ as above, the map 
$\psi_Z^i$ is  holomorphic. 
\item Identifying the tangent space at $0\in J^i(X)$ with $H^{2i-1}/F^i=\overline {F^i }$ one has an inclusion
$
(\psi_Z^i)_* T_{y_o}Y \subset H^{i-1,i}\subset \overline {F^i }$.
\end{enumerate}
\end{lemma}
It follows from  Lemma~\ref{lem:OnHs}.(1) that the  image $J^i_\alg(X)$ of the Abel--Jacobi map is a complex torus. 
Although $J^i(X)$ is not  always an abelian variety, Lemma~\ref{lem:OnHs}.(2) and Observation~\ref{obs:JacIsAb}  imply that
$J^i_\alg(X)$ is an abelian variety.
Next, observe that equation~\eqref{eqn:CruxInc} implies
 that there is a surjective holomorphic map
 \begin{equation}
 \label{eqn:compare}
 \rho^i :  J_\alg ^i(X)=\ch i \alg X \cap  \ker (\Psi_{\rm AJ} )  \longtwoheadrightarrow  \pic^i (X):= \ch i \alg X/ K^i(X)
 \end{equation} 
defining a quotient torus $ \pic^i (X)$ which then also is an abelian variety:
\begin{corr}
$J^i_\alg(X)$ as well as $\pic^i(X)$ are abelian varieties.
\end{corr}
It follows from this that
 \begin{equation}
  \label{eqn:OnGriffConj}
   GC(X,i)   \iff   J_\alg ^i(X) \text{ is isogenous to }  \pic^i (X). 
 \end{equation} 
The notation $\pic^i (-)$, as well as the  terminology  $i$-th \emph{higher Picard variety}  for these,  is motivated by the fact that
 $\pic ^1(X)=\pic^{(0)}(X)$,  the Picard variety of $X$ and $\pic^d (X)$ is the Albanese variety (see Example~\ref{exm:basic}).
 
 \begin{rmk}
 In  Example~\ref{exm:basic}    $\pic^{(0)} (X)$ and $\pic^d(X)$ appeared in connection with  \cite{liebIJ,hsaito}.
 In fact    higher Picard  varieties have  been investigated  in characteristic 0 in   \cite{liebIJ} and 
 in any characteristic in \cite{hsaito}.  
 \end{rmk}

By   \cite[\S 4]{liebIJ}   one has a functoriality property as in equation~\ref{eqn:CorOnJacs}:

\begin{lemma} Let $Y$ be a smooth projective variety and  
$f\in \cors {j-i}  {} {X ,Y} $ an integral correspondence. Then 
there exists precisely one abelian variety homomorphism $f_{\picmap} :\pic^ i(X)  \to \pic^j(Y)$ such that
the diagram
\[
\xymatrix{ 
\ch {i} \alg X  \ar[r]^{f_\ast }                  \ar[d]                      & \ch j \alg Y           \ar[d] \\
                  \pic^{i}(X) \ar[r]_{f_{\picmap} }                        & \pic^j(Y)
}
\]
commutes.
\end{lemma}

While for $J^i_\alg(X)$ duality is not clear for arbitrary $i$, for $\pic^i(X)$ this holds due to an argument of Grothendieck  which is 
explained in \cite[\S 4]{liebIJ}. This  uses a so-called 
\textbf{\emph{Poincaré $i$-correspondence}}\footnote{Also called a Poincaré $i$-cycle (class).}
 $\pi^i$, a correspondence  of degree   $i-e$, $e=\dim {\pic^i (X)}$ 
  from $\pic ^i(X)$ to $X$ such that 
 the induced endomorphism $f_{\picmap}$  of
 $\pic^i(X)$ (see the preceding
 diagram for $i=e $) is multiplication by some integer depending on $(X,i)$.  The transpose of the Poincaré $i$-correspondence is a correspondence
 for $X$ to $\pic^i(X)$ of degree $i-d$. It sends $\ch {d+1-i} \alg  X$ to $\ch {1}\alg X$ and setting $j=d+1-i$,
 it induces a morphism
 $$
(\pi^i)^*= \lambda^j_X: \pic^j (X)\to \pic^1(\pic^i (X))= \pic^{(0)}(\pic^i (X))= (\pic^i(X))^*.
 $$
This gives the announced duality statement:
    
 \begin{thm}[\protect{\cite[\S 4]{liebIJ} and \cite[Theorem (3.6)]{hsaito}}] 
 \label{thm:OnPic} Let $X$ be a $d$-dimensional smooth projective variety.   Let $i,j$ non-negative integers with $i+j=d+1$.   
 Then   
 $$
 \lambda^j_X: \pic^j (X) \to \pic^{i} (X)^*
 $$ 
 is an isogeny which does not depend  of the choice of the Poincaré $i$-correspondence.
  \end{thm}

\begin{corr} \label{cor:OmDualPics}
Suppose that $GC(X,i)$ holds, then if $J^ i  _\alg (X)$ is dual to $J ^j_\alg (X)$  also $GC(X,j)$ holds.
\end{corr}

\section{Relation with the  generalized Hodge conjecture}
\label{sec:GHC}

The coniveau filtration on the cohomology $H^\bullet(X)_\bQ$, $X$ smooth projective,  is defined as follows
\[
N^iH^m (X)= \bigcup_{\codim Z\ge i} \im( H^m_Z(X)\otimes\bQ \to H^m(X)_\bQ),\quad Z\subset X \text{ a subvariety,}
\]
and where
$ 
H^m_Z(X):= H^m(X,X\setminus Z). 
$ 
If $Z$ is smooth, by Poincar\'e--Lefschetz duality  one has an isomorphism
$$
H^m_Z(X)_\bQ\simeq H^{m-2i}(Z)_\bQ.
$$
On the Hodge theoretic side,  $H^{m-2i}(Z)$ has a pure weight $m-2i$ Hodge structure, and so has  level 
at most $m-2i$.\footnote{Recall ll that  the level of a Hodge structure equals  $ \max_{H^{p,q}\not=0} |p-q|$.}
   In case $Z$ is smooth of codimension $i$,  one has
 $$
 \im( H^m_Z(X)\otimes\bQ \to H^m(X)_\bQ)= \im( H^{m-2i}(Z)(-i)\otimes\bQ  \mapright{\iota_*} F^i H^m_\bC\cap H^m(X)_\bQ )
 $$ 
 with $\iota:Z\into X$ the embedding. Note there that for the Gysin map $\iota_*$ a Tate twist has to be applied in order to make it a
 morphism preserving weights. It follows that
   the image of $\iota_*$  also has level   $\le  |m-2i|$. In case $Z$ is singular, one first   forms the  desingularization $\tilde Z$,
and the crucial non-trivial remark here is that the  resulting Gysin map  $H^{m-2i}(\tilde Z)(-i) \to  H^m(X)$  
has the same image as $\iota_*$. This follows from mixed Hodge structures. See e.g. \cite[Lemma 7.2]{mht}.

Summarizing, since $N^iH^m (X) \subset F^i H^m_\bC (X)\cap H^m_\bQ (X)$, the subspace $ N^iH^m (X) $ is a 
$\bQ$-Hodge   structure of level $\le |m-2i|$.  
However, the right hand side $F^i H^m_\bC (X)\cap H^m_\bQ (X)$ is not in general a 
Hodge structure  as explained in \cite{groth} by Grothendieck. One should replace it by
 \begin{equation}
 \label{eqn:GrothHhdg}
H^{m,i}_\hdg(X) =\text{largest Hodge substructure of } F^i H^m_\bC (X)\cap H^m_\bQ (X), 
 \end{equation}  
 which is  the  largest Hodge substructure of  $H^m(X)_\bQ$ of level $\le |m-2i|$. This leads to

\begin{conj}\label{conj:GHC}
GHC$(X,m,i)$: For a smooth complex projective variety $X$ one has  $\ H^{m,i}_\hdg(X) =N^i H^m (X) $.
\end{conj}

 While the classical Hodge conjecture states that $H^{2i}_\alg(X)=H^{2i}_\hdg(X)$,  the conjectures GHC$(X,m,i)$ involve 
 the entire cohomology, not just the even degree part.  
 For odd $m=2i-1$  in the remaining part of this note a special role is played by
 \begin{equation}
H^{2i-1}_\alg(X) := N^{i}H^{2i-1}(X),\quad    H^{2i-1}_\alg(X)_\bC = N^{i}H^{2i-1}(X)\otimes\bC  \label{eqn:algodd}.
\end{equation}
The geometric nature of the coniveau filtration implies that it is compatible with the action of the Lefschetz operator.
In the present situation it gives a homomorphism $L: H^{2i-1}_\alg(X) \to H^{2i+1}_\alg(X) $.
This leads to  the following result.

\begin{corr}  \label{cor:OnConjB} Suppose that  $B(X)$ holds. If $x\in H^{2i-1}_\alg(X) $, 
  then all primitive components of the Lefschetz decomposition belong to $H^{\bullet}_\alg(X)$.
  \end{corr}
  \begin{proof}  As remarked above, for even cohomology this is   Lemma~\ref{lem:OnConjBeven}. 
   Conjecture $B(X)$ states that $\Lambda$ is induced by a correspondence of degree $-1$.   Just as $L$, any correspondence 
   acts on the coniveau filtration. More precisely, a correspondence of rank $r$ from $X$ to $X$ sends 
   $N^i H^n(X)_\bQ$ to $N^{i+ r}H^{n+2r}(X)_\bQ$ and so in particular  it sends $H^{2i-1}_\alg(X)$ to $H^{ 2(i+r)}_\alg(X)$.
  This implies that the  argument used in  the proof of    Lemma~\ref{lem:OnConjBeven}  applies here.
   \end{proof}
 
There is also a relation of the coniveau filtration with  the abelian variety  $J^i_\alg(X)$.  To explain this,
 recall that for  any complex torus $J=V/\Gamma$ the  
 tangent space $T(J)$ at the identity is equal to $V$.  Applying this to  $  J^i(X)$, we get 
  $ T(J^i(X))= H^{2i-1}(X)/F^i=\bar F^{2i -1-i+1}=\bar F^i$.        By Lemma~\ref{lem:OnHs}.(2)
  $ T(J^i(X)) \oplus \overline{T(J^i(X))} \subset F^i\oplus \bar F^i \subset H^{2i-1}(X)$.
  Then $ T(J_\alg^i(X)) \oplus \overline{T(J_\alg^i(X))} $   is a sub Hodge structure, since $T(J_\alg^i(X))$
  is of type $(i-1,i) $, and so the Hodge structure  on $T(J_\alg^i(X))\oplus \overline{T(J_\alg^i(X))}$
  is of level $\le 1$  and   in fact it is the maximal such Hodge structure as shown by Murre: 
 
  \begin{lemma}[\protect{\cite[Lemma 4.3]{murre}}] $H^{2i-1}_\alg (X)_\bC   =T(J_\alg^i(X)) \oplus \overline{T(J_\alg^i(X))}$ and hence
  $$
  J^i_\alg(X)= H^{2i-1}_\alg(X)/ [F^i(H^{2i-1}_\alg(X))+  H^{2i-1}(X)  \cap  H^{2i-1}_\alg(X)  ].
  $$
  \end{lemma}
In particular   $H^{2i-1}_\alg(X)$  is a Hodge substructure of $H^{2i-1}(X)_\bQ$  of level at most 1, while, referring to \eqref{eqn:GrothHhdg},
$$  
H^{2i-1}_\hdg(X) := H^{2i-1,i}_\hdg(X),
$$  is  the largest Hodge substructure of $H^{2i-1} (X)_\bQ$ of level $\le 1$, that is, which is of type $(i-1,i )+(i,i-1)$,
By Observation~\ref{obs:JacIsAb}, it then follows that  the associated  torus 
 \begin{equation}
    \label{eqn:DualHdgJac}
    J ^i_\hdg(X) = H^{2i-1}_\hdg(X)/[F^i(H^{2i-1}_\hdg(X))+  H^{2i-1}(X) \cap  H^{2i-1}_\hdg(X) ]
\end{equation} 
 is an abelian variety. Moreover,   up to isogeny  $J^i_\alg(X)$ is a abelian subvariety  of $J^i_\hdg(X)$.  Hence 
 \begin{obsv}
 $  GHC (X,2i-1,i)\implies J^i_\hdg(X) \text{ is isogenous to }J^i_\alg(X) $.
\end{obsv}
  \begin{rmq}   By Corollary~\ref{cor:OnCX}  below,  $GHC (X,2i-1,i)$   even   implies $GC(X,i)$ (as observed by Griffiths in \cite{griftrans}).
  \end{rmq}

 \begin{exmple} Since $J^1(X)=\pic^{(0)}(X)=J^1_\alg(X)=J^1_\hdg(X)$, 
 for divisors incidence equivalence coincides with Abel-Jacobi equivalence.
 \end{exmple}
 
\section{Bloch's biextension and incidence equivalence}
\label{sec:biexts}

 Let  $T$ be  an abelian variety and  $T^* =\pic ^{(0)} T$ its dual, considered 
 as  the group   under  the addition operation of  divisor classes  on $T$ of homological equivalent to $0$.
There is a  Poincar\'e line bundle $\cP$ on $T \times T^*$ where for every  $y\in T^*$ the restriction $\cP_y:=\cP|_{T\times y}$ 
 is the line bundle corresponding to the divisor class $y$. It can be made unique by demanding that $\cP|_{0\times T^*}$ is trivial.
On  $\bbP=\cP\setminus \{\text{zero section}\}$ the group $\bC^\times $ acts freely making 
$p: \bbP \to T\times T^*$ a $\bC^\times$-torsor over $T\times T ^*$, called the \textbf{\emph{Poincaré-biextension}}. 
Every point  $y\in T^*$ defines an exact sequence
\[
1 \to \bC^\times \to \bbP_y \to T\to 0,\quad \bbP_y= p^{-1}(T\times y),  
\]
that  is, 
\[
\bbP_y\in   \ext ^1 (T,\bC^\times),
\]
an extension  of $T$ by $\bC^\times$ in the category of commutative group varieties over $\bC$.  
In order to  pass  to dual intermediate jacobians, for the
remainder of this section I shall use the convention
$$
(i,j)\in \bN^2  \text{ is a pair of numbers with  }i+j=d+1.
$$
This is motivated by the property that  for  any complex projective manifold $X$ of dimension $d$ the dual of    $ J^i(X)$ is $J^{j}(X)$.
\par
If $\psi^{j}_\aj :\ch{j}  \hom X \to J^{j} (X)$ is the Abel-Jacobi map, and 
$\psi^{j}_\aj (\eta) =y\in J^{j}(X)$, then (almost by definition)
$\bbP_y$ depends only on the class of $\eta$ modulo  Abel-Jacobi equivalence and  thus one writes  $\bbP_\eta$. Hence
\[
\bbP_\eta \in \ext ^1 ( J^i(X),\bC^\times )\implies \psi^j_\aj (\eta) \text{ is a split  extension } \iff \eta\sim_\aj 0.
\]
Bloch~\cite{blochbiext} defined a replacement  for the Poincar\'e torsor for Chow groups homologically equivalent to zero, namely   
a certain $\bC^\times$-torsor, called the \textbf{\emph{Bloch-biextension}}
\[
p: \bE \to \ch i\hom X\times \ch  {j} \hom X .
\]
For $\eta\in \ch {j}  \alg X$,  set   $\bE^\alg_\eta= p^{-1}(\ch i \alg X\times \eta)$.
This defines an extension
$$
1 \to \bC^\times \to \bE^\alg_\eta \to \ch i \alg  X\to 0.
$$ 
The Abel-Jacobi maps $\psi^i$ and $\psi^j$ induce  $\psi : \ch  i \alg X\times \ch {j} \alg X \to J^i(X)\times J^{j}(X) $, and 
 Müller-Stach~\cite[Theorem 1]{mstach} proves  
 that  one has an equality of $\bC^\times$-biextensions
 $$
 \psi^* ( \bbP )=\bE.
 $$ 
 As a result  he shows -- using the algebraic nature of the Bloch-biextension --
  \footnote{Consider also the suggested corrections in \cite[Satz 3.2.1]{meyerBiExt}, \cite[Lemma 5.1, 5.2]{naka}).}
\begin{thm}[\protect{  \cite[Theorem 2]{mstach}}]
$\bE^\alg_\eta\in \ext(  \ch i   {\alg}  X, \bC^\times )  \text{ is a split  extension } \iff \eta\sim_\inc 0$.
\end{thm}
   
\begin{corr}[\protect{  \cite[Theorem 3]{mstach}}] Let   $X$ and $\eta\in \ch  {j } \alg X$ be as before. 
Observe that  $\psi_\aj(\eta) \in  J^j_\alg(X) $.     Assume that moreover 
$$
(*) \qquad \qquad\text{ for some integer } N     \text{ one has }  \psi_\aj(N\eta) \in  (J^i_\alg(X))^*  . \hfill
$$
 Then,  if $\bE^\alg_\eta$ is a  split extension, also $\bbP_{N\eta}$ is a split extension. 

Hence, if  $(*) $  holds, then   $\eta\sim_\inc 0\implies N\eta \sim_\aj 0$.
   \end{corr}
   
   This again poses the problem: how to show that an  element 
 in   $   J^{j}_\alg(X)$ belongs to the dual of $J^i_\alg (X)$?
   This can be decided on the level of tangent spaces:
    if $H^{2i-1}_\alg(X)_\bC= T^i\oplus \bar T^i$ where $T^i$ is the tangent space of $J^i_\aj (X)$.
   Then $H^{2i-1}_\alg(X)$  is a polarized Hodge structure of type $(i -1,i)+(i,i-1)$ and almost by definition  of the dual torus, this yields  
    the following  criterion. 
 
  \begin{crit} \label{crit:InIsAj}
        $J^{j}_\alg (X)$ belongs to the dual of $J^i_\alg(X)$ if
       the cupproduct pairing $H^{2i-1}_\alg(X)  \otimes H^{2j-1 }_\alg(X) \to \bQ$ is non-degenerate on the first factor.
    Hence if this is the case, $GC(X,i)$ holds.
   \end{crit}
 
 \begin{rmk} Conjecture $D(X)$ is equivalent to the assertion that  the cup pairing pairing  
   $H^{2\bullet}_\alg(X) \times H^{2d-2\bullet}_\alg(X)$ is non-degenerate (see \cite[Proposition 3.6]{kleiman})
   and so Criterion~\ref{crit:InIsAj} is a partial  odd-rank version of conjecture  $D(X)$.
\end{rmk}

The dual of   $J^i_\hdg (X)$ as defined in \eqref{eqn:DualHdgJac}
is $J^j _\hdg(X)$. This is because the polarization on the maximal sub Hodge structures  
$H^{2\bullet-1}_\hdg(X)  \subset H^{2\bullet-1}(X)_\bQ $
of level $\le 1$ gives a duality pairing
$H^{2i-1}_\hdg(X) \times H^{2d-2i+1}_\hdg(X) \to \bQ$.
Hence, one has:

   \begin{corr} \label{cor:OnCX} Conjecture $ GHC (X,2i-1,i)$ implies  $GC(X,i)$.
   \end{corr}

\begin{rmk}   
 There is a proof of Corollary~\ref{cor:OnCX} in 
 \cite[p. 178--181]{grifbombay}. The proof given there is  based on explicit calculations with forms and integrals,  while the above proof is direct.  
\end{rmk}

 The proof of $GC(X,2)$ as given in \cite{murre} first of all uses   that   $H^{3}_\alg (X)= H^3_{\prim,\alg}(X)+ L H^1_{\prim,\alg}(X)$ 
 which is the case  since $   H^1_{\prim,\alg}=H^1(X)_\bQ$   by Abel's theorem. This generalizes  to the following condition which is not known to be true in general.
  \begin{cond*}[$\Lambda,X,i$]  Let $v\in H^{2i-1 }_\alg (X)$ and let   $v= \sum_{a=s}^ {s+t}   L^av_{2(i-a) -1 }$ be its Lefschetz decomposition.
Then all $ v_{2(i -a)-1} \in H^{2(i-a)- 1  }_\alg (X)$.
\end{cond*}

\begin{rmk} Corollary~\ref{cor:OnConjB} implies that the above condition is implied by conjecture $B(X)$.
\end{rmk} 

The geometric proof of $GC(X,2)$ given in \cite[\S 6]{murre}  further   uses  that 
the  condition given in Criterion~\ref{crit:InIsAj} holds.
In this case it states that the intersection pairing $H^3_\alg(X) \times H^3_\alg(X)$ is non-degenerate  
 and is stated   in loc. cit. as Lemma~5.2. However it seems that  its  proof   is not complete:    the term $L \xi_1^{0,1}\cup \bar \xi_3^{1,2}\cup (u^{d-3})$ occurring in the computation  does not evidently vanish as claimed since it   is of type $(d,d)$.
 \par
 Assuming $(\Lambda,X,i)$  I propose a simpler approach (for any $(d,i)$) using the $\ast$-operator which is a linear operator on the total 
 real  de Rham cohomology. Almost by  definition  it  has  has the crucial property that
 for any class $v\in H^k(X)_\bR$ the class $v\wedge \ast v$ is of the form $f dV$, where $V$ is  a volume form and $f: X\to \bR_{\ge 0}$ a 
 smooth function. Hence its integral (which is the self-intersection of the class $v$) is non-negative. In fact, 
\[
 v\cdot v= \int _X v\wedge \ast v\ge 0, \text{ and equality holds } \iff v=0.
 \]
In the present setting  $v$ has a   primitive decomposition $v= \sum_{r\ge 0}  L^r v_{k-2r}\in H^k(X)_\bR$.
The formula  for the behaviour of the $\ast$-operator under primitive decomposition is     given   on p. 76 on \cite{weilboek}. It uses
 the Weil operator $C$ and the Lefschetz operator $L$. 
In somewhat modified  form it reads as follows.\footnote{Weil's results  are far from trivial and use a detailed calculative  approach to Lefschetz' theory instead of  the standard one  using  representation theory of $\slg{\bC}{2}$.} 
\begin{equation*}
\ast (L^r v_{k-2r}) = (-1)^{\half k(k+1)-r}   \cdot a_{d,r} L^{d-k+r } C(v_{k-2r}),\quad  a_{d,r} \in \bQ_{>0},
\end{equation*} 
as long as  $ d-k+r\ge 0$ (otherwise one gets $0$). 
% This equation shows that the $\ast$-operator is algebraic and so, if $v\in H^{k}_\alg(X) $, then   $\ast v\in H^{2d-k}_\alg(X) $.
 Applying this  with  $k=2i-1$, assuming $(\Lambda,X,i)$, by linearity of the $\ast$-operator
 and the fact that $LH^q_\alg(X)\subset H^{q+2}_\alg(X)$,  this implies 
 \begin{equation}\label{eqn:weil}
 \text{if } v\in H^{2i-1}_\alg(X) \text{ then }   \ast v   \in H^{2d -2i+1}_\alg(X)  \text{ and }     v \cdot (\ast v)> 0  \text{ unless  } v=0.
 \end{equation} 
 This then is the major step to prove the following result.

\begin{corr} \label{cor:GenMurre} Condition~$ (\Lambda,X,i)$  implies $GC(X,i) $.  
\end{corr}
\begin{proof}
 It suffices to show   that $(\Lambda,X,i)$ indeed implies that 
the cupproduct pairing 
$$
H^{2i-1 }_\alg(X) \otimes H^{2d-2i+1}_\alg(X)  \to \bQ
$$
 is non-degenerate on the first factor. This follows from the assertion  \eqref{eqn:weil}.
\end{proof}

In fact, in \cite[p. 87--88]{grifschmid} the same idea  is  used  in a simpler situation  where criterion~\ref{crit:InIsAj}    applies for obvious reasons:
\begin{corr} Suppose $X$ is a smooth complex projective variety of odd dimension $d=2m+1$ such that $b_{2k+1}(X)=0$ for $k\not=m$.
Then incidence equivalence and Abel-Jacobi equivalence coincide for codimension $m+1$ cycles     algebraically equivalent to 0.
Examples include smooth complete intersections of dimension $2m+1$  in projective space.
\end{corr}

\begin{rmk} As to the dual of $J^i_\alg(X)$, Condition~$(\Lambda,X,i)$ is only known to give the inclusion $J^{j}_\alg (X)\subset J^i_\alg(X)^*$.
The reverse inclusion follows from Condition~$(\Lambda,X,j)$. So in particular $J^{j}_\alg (X)=J^i_\alg(X)^*$  if 
Condition~$(\Lambda,X,i)$ holds in case  $d=2i-1$, since then $j=i$. 
\end{rmk}

%This is a variant of the Lieberman condition for the conjecture $D(X)$ to hold. This is probably also falling under
%the validity of this conjecture.
    
Let me finish by giving a summary of the results in this section.

\begin{summar}\label{summar:main}
Let $X$ be a $d$-dimensional smooth complex projective variety. Then
 \begin{enumerate}
 \item In the notation introduced in \S~\ref{sec:AJ},   $GC(X,i) \iff 
 J^i_\alg(X)$ is isogenous to $\pic^i (X)$, the generalized Picard variety introduced in \S~\ref{sec:AJ}.
   Both would follow if  Grothendieck's Hodge conjecture  $GHC (X,2i-1,i) $ holds,
 
 \item  $B(X)\implies \Lambda(X,i)  \implies   GC(X,i) $,
 In particular one has $GC(X,i)$ for abelian varieties $X$,
 \item $\Lambda(X,2)$ holds and hence $G(X,2)$ holds,
 \item For $i+j=d+1$ one has $\Lambda(X,i)+ \Lambda(X,j) \iff GC(X,i) +GC(X,j) $. In particular  $\Lambda(X,i)  \iff GC(X,i) $ if   $d+1=2i$,
  \item $GC(X,1)$ and $G(X,d)$  hold (cf. Example~\ref{exm:basic}),
  
  \item $GC(X,m)$ holds if $\dim X=2m+1$ and $b_{2k+1}(X)=0$ for $k\not=m$ such as for complete intersections 
  of dimension $2m+1$ in projective space.
\end{enumerate}

\end{summar}
 
\bibliographystyle{siam}

	%\bibliography{/Users/chrismacbook/SurfDriveLU/Werk/Bib.f/QFbib.bib}

\end{document}